\documentclass[11pt]{article}

\usepackage{amsfonts,amssymb,amsthm,amsmath}
\usepackage{IEEEtrantools,mathrsfs}
\usepackage{paralist}
\usepackage{graphics,graphicx}
\usepackage{epsfig,psfrag}
\usepackage{url}

 \usepackage[colorlinks=true]{hyperref}
\hypersetup{urlcolor=blue, citecolor=red}

\newtheorem{theorem}{Theorem}[section]

\newtheorem{lemma}[theorem]{Lemma}
\newtheorem{proposition}{Proposition}

\theoremstyle{definition}
\newtheorem{definition}[theorem]{Definition}
\newtheorem{remark}{Remark}

\def\N{\mathbb{N}}
\def\cal{\mathcal}
\def\R{\mathbb{R}}
\def\calA{\mathcal{A}}
\def\calE{{\mathcal{E}}}
\def\barcalE{{\overline{\mathcal{E}}}}

\def\2toro{\mathbb{T}^2}
\def\trasl2toro{\overline{\mathbb{T}}^2}

\def\dmint{\tilde{d}_{min}}

\def\ov{\overline}

\begin{document}

\begin{center}\LARGE
  The evolution of the orbit distance in the
  double averaged restricted 3-body problem\\
  with crossing singularities
\end{center}
\bigskip

\centerline{\bf\large Giovanni F. Gronchi$^\dagger$, Chiara
  Tardioli$^{\dagger\ddagger}$}
\bigskip 
\centerline{\large $^\dagger$Dipartimento di Matematica, Universit\`a
  di Pisa (Italy)}
\medskip
\centerline{$^\ddagger$D\'epartement de Math\'ematique (naXys), FUNDP (Belgium)}
\bigskip 
\centerline{\large\texttt{gronchi@dm.unipi.it}}%
\medskip
\centerline{\large\texttt{tardioli@mail.dm.unipi.it}}

\bigskip\bigskip

\begin{abstract}
  We study the long term evolution of the distance between two Keplerian
  confocal trajectories in the framework of the averaged restricted 3-body
  problem. The bodies may represent the Sun, a solar system planet and an
  asteroid.  The secular evolution of the orbital elements of the
  asteroid is computed by averaging the equations of motion over the mean
  anomalies of the asteroid and the planet.  When an orbit crossing with the
  planet occurs the averaged equations become singular. However, it is
  possible to define piecewise differentiable solutions by extending the
  averaged vector field beyond the singularity from both sides of the orbit
  crossing set
  \cite{GM98}, \cite{gronchi_celmecIII}.  In this paper we improve the
  previous results, concerning in particular the singularity extraction
  technique, and show that the extended vector fields are
  Lipschitz-continuous.  Moreover, we consider the distance between the
  Keplerian trajectories of the small body and of the planet.  Apart from
  exceptional cases, we can select a sign for this distance so that it becomes
  an analytic map of the orbital elements near to crossing configurations
\cite{gronchi_tommei}.  We
  prove that the evolution of the `signed' distance along the averaged vector
  field is more regular than that of the elements in a neighborhood of
  crossing times.  A comparison between averaged and non-averaged evolutions
  and an application of these results are shown using orbits of near-Earth
  asteroids.
\end{abstract}

\section{Introduction}

The distance between the trajectories of an asteroid (orbiting around
the Sun) and our planet gives a first indication in the search for
possible Earth impactors.  We call it orbit distance and denote it by
$d_{min}$.\footnote{It is often called MOID (Minimum Orbit
  Intersection Distance, \cite{bowell}) by the astronomers.} A
necessary condition to have a very close approach or an impact with
the Earth is that $d_{min}$ is small. Provided close
approaches with the planets are avoided, the perturbations caused by the
Earth make the asteroid trajectory change slowly with time.  Moreover,
the perturbations of the other planets produce small changes in both
trajectories. The value of the semimajor axis of both is kept constant
up to the first order in the small parameters (the ratio of the mass
of each perturbing planet to the mass of the Sun).  All these effects
are responsible of a variation of $d_{min}$.
We can study the evolution of the asteroid
in the framework of the restricted 3-body
problem: Sun, planet, asteroid. Then it is easy to include more than
one perturbing planet in the
model, in fact the potential energy can be written as sum of terms each
depending on one planet only.

If the asteroid has a close encounter with some planet, the perturbation of the latter
generically produces a change in the semimajor axis of the asteroid.
This can be estimated,
and depends on the mass of the planet, the unperturbed planetocentric
velocity of the small body and the impact parameter, see \cite{resret}.

The orbits of near-Earth asteroids (NEAs, i.e. with perihelion distance $\leq
1.3$ au)\footnote{1 au (astronomical unit) $\approx $ 149,597,870 Km} are
chaotic, with short Lyapounov times (see~\cite{whipple95}), at most a few
decades.  After that period has elapsed, an orbit computed by numerical
integration and the true orbit of the asteroid are practically unrelated and
we can not make reliable predictions on the position of the asteroid.  For
this reason the averaging principle is applied to the equations of motion: it
gives the average of the possible evolutions, which is useful in a statistical
sense.  However, the dynamical evolution often forces the trajectory of a NEA
to cross that of the Earth. This produces a singularity in the averaged
equations, where we take into account every possible position on the
trajectories, including the collision configurations.

The problem of {\em averaging on planet crossing orbits} has been studied in
\cite{GM98} for planets on circular coplanar orbits and then generalized in
\cite{gronchi_celmecIII} including nonzero eccentricities and inclinations of
the planets.
The work in \cite{GM98} has been used to define proper elements for NEAs,
that are integrals of an approximated problem, see \cite{GM01}.
In this paper we compute the main singular term
by developing the distance between two points, one on
the orbit of the Earth and the other on that of the asteroid, at its minimum
points.  This choice improves the results in \cite{GM98},
\cite{gronchi_celmecIII},
where a development at the mutual nodes was used, because it avoids the
artificial singularity occurring for vanishing mutual inclination of the two
orbits.  Moreover, we show that the averaged vector field admits two
Lipschitz-continuous extensions from both sides of the orbit crossing set (see
Theorem~\ref{der_jumps}), which is useful for the numerical computation
of the solutions.

The orbit distance $d_{min}$ is a singular function of the (osculating)
orbital elements when the trajectories of the Earth and the asteroid
intersect. However, by suitably choosing a sign for $d_{min}$ we obtain a
map, denoted by $\tilde{d}_{min}$, which is analytic in a neighborhood of most
crossing configurations (see \cite{gronchi_tommei}).

Here we prove that, near to crossing configurations, the averaged evolution of
$\tilde{d}_{min}$ is more regular than the averaged evolution of the elements,
which are piecewise differentiable functions of time.

The paper is organized as follows.  Section~\ref{s:orbdist} contains
some preliminary results on the orbit distance.
In Sections~\ref{s:avereq},
\ref{s:extract}, \ref{s:gensol} we introduce the averaged equations,
present the results on the singularity extraction method and give the
definition of the generalized solutions, which go
beyond crossing singularities.  In Section~\ref{s:dminevol} we prove the
regularity of the secular evolution of the orbit
distance. Section~\ref{s:numexp} is devoted to numerical experiments: we
describe the algorithm for the computation of the generalized solutions and
compare the averaged evolution with the solutions of the full equations of
motion.  We also show how this theory can be applied to estimate
Earth crossing times for NEAs.

\section{The orbit distance}
\label{s:orbdist}

Let $(E_j,v_j)$, $j=1,2$ be two sets of orbital elements of two
celestial bodies on confocal Keplerian orbits.  Here $E_j$ describes
the trajectory of the orbit and $v_j$ is a parameter along the trajectory,
e.g. the true anomaly.  We denote by $\calE=(E_1,E_2)$ the two-orbit
configuration, moreover we set $V=(v_1,v_2)$.  In this paper we
consider bounded trajectories only.

Choose a reference frame, with origin in the common focus, and write
$\mathcal{X}_j = \mathcal{X}_j(E_j,v_j)$, $j=1,2$ for the Cartesian
coordinates of the two bodies.

For a given two-orbit configuration $\calE$, we introduce the {\em Keplerian
  distance function} $d$, defined by
\[
\2toro\ni V \mapsto d(\calE,V)= |{\mathcal X}_1-{\mathcal X}_2|\,,
\]
where $\2toro$ is the two-dimensional torus and $|\cdot|$ is the Euclidean
norm.

The local minimum points of $d$ can be found by computing all the critical
points of $d^2$. For this purpose in \cite{gronchi02}, \cite{gronchi05},
\cite{kholvass}, \cite{balukhol} the authors have used methods of
computational algebra, such as resultants and Gr\"obner's bases,
which allow us to compute efficiently all the solutions.

Apart from the case of two concentric coplanar circles, or two
overlapping ellipses, the function $d^2$ has finitely many stationary
points. There exist configurations attaining 4 local minima of $d^2$:
this is thought to be the maximum possible, but a proof is not known
yet. A simple computation shows that, for non-overlapping
trajectories, the number of crossing points is at most two,
see~\cite{gronchi05}.

Let $V_h=V_h(\calE)$ be a local minimum point of $V \mapsto
d^2(\calE,V)$. We consider the maps
\[
\calE \mapsto d_h(\calE) = d(\calE,V_h)\,,
\hskip 1cm
\calE \mapsto d_{min}(\calE) = \min_h d_h(\calE)\ .
\]
For each choice of the two-orbit configuration $\calE$, $d_{min}(\calE)$ gives
the orbit distance.

The maps $d_h$ and $d_{min}$ are singular at crossing configurations, and
their derivatives do not exist.
 We can deal with this singularity and obtain
analytic maps in a neighborhood of a crossing configuration $\calE_c$ by
properly choosing a sign for these maps.
We note that $d_h$, $d_{min}$ can present singularities without orbit
crossings.  The maps $d_h$ can have bifurcation singularities, so that the
number of minimum points of $d$ may change. Therefore the maps $d_h$,
$d_{min}$ are defined only locally.  We say that a configuration $\calE$ is
non-degenerate if all the critical points of the Keplerian distance function
are non-degenerate.  If $\calE$ is non-degenerate, there exists a neighborhood
${\cal W}$ of $\calE\in\R^{10}$ such that the maps $d_h$, restricted to ${\cal
  W}$, do not have bifurcations.  On the other hand, the map $d_{min}$ can
lose regularity when two local minima exchange their role as absolute minimum.
There are no additional singularities apart from those mentioned above.  The
behavior of the maps $d_h$, $d_{min}$ has been investigated in
\cite{gronchi_tommei}.  However, a detailed analysis of the occurrence of
bifurcations of stationary points and exchange of minima is still lacking.

Here we summarize the procedure to deal with the crossing
singularity of $d_h$; the procedure for $d_{min}$ is the same.
We consider the points on the two orbits corresponding
to the local minimum points $V_h = (v_1^{(h)}, v_2^{(h)})$ of $d^2$:
\[
\mathcal{X}_1^{(h)} = \mathcal{X}_1(E_1,v_1^{(h)})\,;\hskip 1cm
\mathcal{X}_2^{(h)} = \mathcal{X}_2(E_2,v_2^{(h)})\ .
\]
We introduce the vectors tangent to the trajectories $E_1, E_2$ at these
points
\[
\tau_1^{(h)} = \frac{\partial {\mathcal X}_1}{\partial v_1}(E_1,v_1^{(h)})\,,
\hskip 1cm
\tau_2^{(h)} = \frac{\partial {\mathcal X}_2}{\partial v_2}(E_2,v_2^{(h)})\,,
\]
and their cross product
\[
\tau_3^{(h)} = \tau_1^{(h)}\times\tau_2^{(h)}\ .
\]
We also define
\[
\Delta = \mathcal{X}_1 - \mathcal{X}_2\,,
\hskip 1cm
\Delta_h = \mathcal{X}_1^{(h)} - \mathcal{X}_2^{(h)}\ .
\]
The vector $\Delta_h$ joins the points attaining a local minimum of $d^2$ and
$|\Delta_{h}|=d_h$.

\begin{figure}[t]
\psfragscanon
\psfrag{t1}{$\tau_1^{(h)}$}\psfrag{t2}{$\tau_2^{(h)}$}\psfrag{t3}{$\tau_3^{(h)}$}
\psfrag{Dmin}{$\Delta_h$}
\centerline{\epsfig{figure=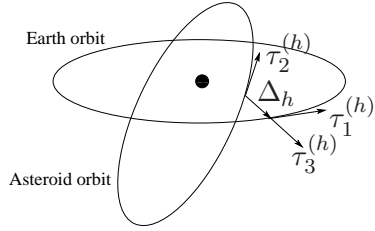,width=4.5cm}}
\psfragscanoff
\caption{Geometric properties of the critical points of $d^2$ and regularization
  rule.}
\label{regrule}
\end{figure}
From the definition of critical points of $d^2$ both the vectors
$\tau_1^{(h)}$, $\tau_2^{(h)}$ are orthogonal to $\Delta_h$, so that
$\tau_3^{(h)}$ and $\Delta_h$
are parallel, see Figure~\ref{regrule}.  Denoting by $\hat{\tau}_3^{(h)}$,
$\hat{\Delta}_h$ the corresponding unit vectors and by a dot the Euclidean
scalar product,
the distance with sign
\begin{equation}
\tilde{d}_h = \bigl(\hat{\tau}_3^{(h)}\cdot\hat{\Delta}_h\bigr)\, d_h
\label{choosesign}
\end{equation}
is an analytic function in a neighborhood of most crossing configurations.
Indeed, this smoothing procedure fails at crossing configurations such that $\tau_1^{(h)}$, $\tau_2^{(h)}$ are parallel.  A
detailed proof can be found in \cite{gronchi_tommei}.  Note that, to obtain
regularity in a neighborhood of a crossing configuration, we lose continuity
at the configurations with $\tau_1^{(h)}\times\tau_2^{(h)}=0$ and
$d_h\neq 0$.

The derivatives of $\tilde{d}_h$ with respect to each component $\calE_k$,
$k=1\ldots 10$ of $\calE$ are given by
\begin{equation}
\frac{\partial \tilde{d}_h}{\partial \calE_k} =
\hat{\tau}_3^{(h)}\cdot \frac{\partial \Delta}{\partial \calE_k} (\calE,
V_h)\ .
\label{derdhtilde}
\end{equation}

We shall call (signed) orbit distance the map $\tilde{d}_{min}$.


\section{Averaged equations} 
\label{s:avereq}

Let us consider a restricted 3-body problem with the Sun, the Earth and an
asteroid.  The motion of the 2-body system Sun-Earth is a known function of
time.
We denote by $\mathcal{X}, \mathcal{X}'\in \R^3$ the heliocentric position of
the asteroid and the planet respectively. The equations of motion for the
asteroid are
\begin{equation}
\ddot{\mathcal{X}} = -k^2\frac{\mathcal{X}}{|\mathcal{X}|^3} + \mu k^2\Bigl[
\frac{\mathcal{X}'-\mathcal{X}}{|\mathcal{X}'-\mathcal{X}|^3} -
\frac{\mathcal{X}'}{|\mathcal{X}|^3}\Bigr]\,,
\label{restr3b}
\end{equation}
where $k$ is Gauss' constant and $\mu$ is a small parameter representing the
ratio of the Earth mass to the mass of the Sun.

We study the motion using Delaunay's elements ${\cal Y} =
(L,G,Z,\ell,g,z)$, defined by
\begin{displaymath}
\begin{array}{l}
        L = k\sqrt{a}\,,              \cr
        G = k\sqrt{a(1-e^2)}\,,       \cr
        Z = k\sqrt{a(1-e^2)}\cos I\,, \cr
\end{array}
\hskip 2cm
\begin{array}{l}
        \ell = n (t-t_0)\,,\cr
        g = \omega\,,      \cr
        z = \Omega\,,      \cr
\end{array}
\end{displaymath}
where $(a,e,I,\omega,\Omega,\ell)$ are Keplerian elements, $n$ is the mean
motion and $t_0$ is the time of passage at perihelion.  Delaunay's elements of
the Earth are denoted by $(L',G',Z',\ell',g',z')$.  We write $\calE = (E,E')$
for the two-orbit configuration, where $E,E'$ are Delaunay's elements of
the asteroid and the Earth respectively.
Using the canonical variables ${\cal Y}$, equations \eqref{restr3b}
can be written in Hamiltonian form as
\begin{equation}
  \dot{\cal Y}  = \mathbb{J}_3\,\nabla_{\cal Y} H\,,
\label{Heqs}
\end{equation}
where we use
\begin{displaymath}
\mathbb{J}_n =
\left[\begin{array}{cc}
\ {\cal O}_n   &-{\cal I}_n \\
{\cal I}_n &{\cal O}_n
\end{array}\right]\,,
\hskip 1cm n\in\N\,,
\end{displaymath}
for the symplectic identity matrix of order $2n$. The Hamiltonian
\[
H = H_0 - R
\]
is the difference of the two-body (asteroid, Sun) Hamiltonian
\[
H_0 = -\frac{k^4}{2L^2}
\]
and the perturbing function
\[
R = \mu k^2\biggl(\frac{1}{|{\cal X}-{\cal X}'|} -
\frac{{\cal X}\cdot{\cal X}'}{|{\cal X'}|^3}\biggr)\,,
\]
with ${\cal X}, {\cal X}'$ considered as functions of ${\cal Y}, {\cal Y}'$.

The function $R$ is the sum of two terms: the first is the
direct part of the perturbation, due to the attraction of the Earth
and singular at collisions with it. The second is called indirect
perturbation, and is due to the attraction of the Sun on the Earth.

We can reduce the number of degrees of freedom of \eqref{Heqs} by averaging
over the fast angular variables $\ell,\ell'$, which are the mean anomalies of
the asteroid and the Earth.  As a consequence, $\ell$
becomes a cyclic variable, so that the semimajor axis $a$
is constant in this simplified
dynamics.  For a full account on averaging methods in Celestial Mechanics see
\cite{AKN}.

The averaged equations of motion for the asteroid are given by
\begin{equation}
  \dot{\ov Y}  = -\mathbb{J}_2\,{{\overline{\nabla_Y R}\,}}\,,
\label{averrhs}
\end{equation}
where $Y = (G,Z,g,z)^t$, ${\ov Y} = ({\ov G},{\ov Z},
{\ov g},{\ov z})^t$ are some of Delaunay's elements, and
\begin{displaymath}
\overline{\nabla_Y R} = \frac{1}{(2\pi)^2}
\int_{\2toro} \nabla_Y R\,d\ell\,d\ell'\,,
\end{displaymath}
with $\2toro=\{(\ell,\ell'): -\pi\leq\ell\leq\pi, -\pi\leq\ell'\leq\pi\}$,
is the vector of the averaged partial derivatives of the perturbing
function $R$.
Equation~\eqref{averrhs} corresponds to the scalar equations
\[
\dot{\ov G} = {\displaystyle\frac{\ov{\partial R}}{\partial g}} \,,
\qquad
\dot{\ov Z} = {\displaystyle\frac{\ov{\partial R}}{\partial z}} \,,
\qquad
\dot{\ov g} = -{\displaystyle\frac{\ov{\partial R}}{\partial G}} \,,
\qquad
\dot{\ov z} = -{\displaystyle\frac{\ov{\partial R}}{\partial Z}}\ .
\]

We can easily include more planets in the model. In this case the
perturbing function is sum of terms $R_i$, each depending on the coordinates
of the asteroid and the planet $i$ only, with a small parameter $\mu_i$,
representing the ratio of the mass of planet $i$ to the mass of the Sun.

Note that, if there are mean motion resonances of low order with the planets,
then the solutions
of the averaged equations \eqref{averrhs} may be not representative of the
behavior of the corresponding components in the solutions of \eqref{Heqs}.

Moreover, when the planets are assumed to move on circular coplanar orbits we
obtain an integrable problem. In fact the semimajor axis $a$, the
component $Z$ of the angular momentum orthogonal to the invariable
plane\footnote{Here we mean the common plane of the planetary trajectories.}
and the averaged Hamiltonian ${\ov H}$ are first integrals
generically independent and in involution (i.e. with vanishing
Poisson's brackets).
Taking into account the eccentricity and the inclination of the planets
the problem is not integrable any more.

In \cite{kozai62} the secular evolution of high eccentricity and inclination
asteroids is studied in a restricted 3-body problem, with Jupiter on a
circular orbit. Nevertheless, crossings with the perturbing planet are excluded in that
work.  In \cite{lidov} there is a similar secular theory for a satellite of
the Earth. The dynamical behavior described in
\cite{kozai62}, \cite{lidov} is usually called {\em Lidov-Kozai mechanism}
in the literature and
an explicit solution to the related equations is given in \cite{kinonakai07}.

If no orbit crossing occurs, by the theorem of differentiation under the
integral sign the averaged equations of motion (\ref{averrhs}) are equal to
Hamilton's equations
\begin{equation}
  \dot{\ov Y} = -\mathbb{J}_2\, \nabla_Y \overline{R}
\label{rhsRbar}
\end{equation}
where
\[
{\ov R} = \frac{1}{(2\pi)^2} \int_{\2toro} R\,d\ell\,d\ell' =
  \frac{\mu k^2}{(2\pi)^2}\int_{\2toro} \frac{1}{|{\cal X}-{\cal
      X}'|}\,d\ell\,d\ell'
\]
is the averaged perturbing function. The average of the indirect term of $R$
is zero.

When the orbit of the asteroid crosses that of the Earth a singularity
appears in \eqref{averrhs}, corresponding to the existence of a collision
for particular values of the mean anomalies.
We study this singularity to define generalized solutions
of \eqref{averrhs} which go beyond planet crossings. Since the
semimajor axis of the asteroid is constant in the averaged dynamics, we expect
that the generalized solutions can be reliable only if there are no close
approaches with the planet in the dynamics of equations~\eqref{Heqs}.

\section{Extraction of the singularity}
\label{s:extract}

In the following we denote by $\calE_c$ a non-degenerate crossing
configuration with only one crossing point, and we choose the minimum point
index $h$ such that $d_h(\calE_c) = 0$. For each $\calE$ in a neighborhood of
$\cal E_c$ we consider Taylor's development of $V\mapsto d^2({\cal E},V)$,
$V=(\ell,\ell')^t$, in a neighborhood of the local minimum point $V_h=V_h(\calE)$:
\begin{equation}
d^2({\cal E},V) = d_h^2({\cal E}) + \frac{1}{2}(V-V_h)\cdot{\cal
  H}_h({\cal E})(V-V_h) + {\cal R}_3^{(h)}({\cal E},V)\,,
\label{taylor_d2}
\end{equation}
where
\[
{\cal H}_h({\cal E}) = \frac{\partial^2 d^2}{\partial
  V^2}(\calE,V_h(\calE))
\]
is the Hessian matrix of $d^2$ in $V_h=(\ell_h,\ell_h')^t$, and
\begin{eqnarray}
{\cal R}_3^{(h)}({\cal E},V) &=&\sum_{|\alpha|=3} r_\alpha^{(h)}({\cal
  E},V)(V-V_h)^\alpha\,, \label{intform}\\
r_\alpha^{(h)}({\cal E},V) &=& \frac{3}{\alpha!}
\int_0^1(1-t)^2 D^\alpha d^2(\calE,V_h+t(V-V_h))\,dt \label{r_alpha}
\end{eqnarray}
is Taylor's remainder in the integral form.\footnote{In
  \eqref{intform}, \eqref{r_alpha}
  $\alpha=(\alpha_1,\alpha_2)\in(\N\cup\{0\})^2$ is a multi-index,
  hence
\[
|\alpha| = \alpha_1+\alpha_2\,,
\hskip 0.5cm
\alpha! = \alpha_1!\alpha_2!\,,
\hskip 0.5cm
V^\alpha = v_1^{\alpha_1} v_2^{\alpha_2}\,,
\hskip 0.5cm
D^\alpha f = \frac{\partial^{|\alpha|} f}
{\partial v_1^{\alpha_1}\partial v_2^{\alpha_2}}\,,
\]
for a vector $V=(v_1,v_2)$ and a smooth function $V\mapsto f(V)$.}
We introduce the approximated distance
\begin{equation}
\delta_h = \sqrt{d_h^2 + (V-V_h)\cdot {\cal A}_h(V-V_h)}\,,
\label{approxdist}
\end{equation}
where
\[
\calA_h = \frac{1}{2}{\cal H}_h = \left[
\begin{array}{cc}
|\tau_h|^2+ \displaystyle\frac{\partial^2\mathcal{X}}{\partial \ell^2}(E,\ell_h)\cdot\Delta_h
                   &-\tau_h\cdot\tau_h' \cr
                   &\cr
-\tau_h\cdot\tau_h' &|\tau_h'|^2 -\displaystyle \frac{\partial^2\mathcal{X}'}{\partial \ell'^2}(E',\ell_h')\cdot\Delta_h \cr
\end{array}
\right],
\]
and
\[
\Delta_h = \Delta_h(\calE)\,,
\qquad
\tau_h = \frac{\partial {\mathcal X}}{\partial \ell}(E,\ell_h)\,,
\quad
\tau_h' = \frac{\partial {\mathcal X}'}{\partial \ell'}(E',\ell_h')\ .
\]

\begin{remark} If the matrix ${\cal A}_h$ is non-degenerate, then it is
  positive definite because $V_h$ is a minimum point, and this property holds
  in a suitably chosen neighborhood ${\cal W}$ of $\calE_c$.  At a
  crossing configuration $\calE=\calE_c$ the matrix ${\cal A}_h$ is degenerate
  if and only if the tangent vectors $\tau_h$, $\tau_h'$ are parallel
  (see~\cite{gronchi_tommei}):
\[
\det{\cal A}_h(\calE_c)=0
\quad \Longleftrightarrow \quad
\tau_h(\calE_c)\times \tau_h'(\calE_c) = 0\ .
\]
\end{remark}

First we estimate the remainder function $1/d - 1/\delta_h$.
To this aim we need the following:
\begin{lemma}
  There exist positive constants $C_1$, $C_2$ and a neighborhood ${\cal U}$ of
  $(\calE_c, \\
  V_h(\calE_c))$ such that
\begin{equation}
C_1 \delta_h^2 \leq d^2 \leq C_2 \delta_h^2
\label{d2_est}
\end{equation}
holds for $(\cal E, V)$ in ${\cal U}$. Moreover, there exist positive
constants $C_3$, $C_4$ and a neighborhood ${\cal W}$ of $\calE_c$ such that
\begin{equation}
d_h^2 + C_3\vert V-V_h\vert^2 \leq \delta_h^2 \leq
d_h^2 + C_4\vert V-V_h\vert^2
\label{norm_equiv}
\end{equation}
holds for $\calE$ in ${\cal W}$ and for every $V\in\2toro$.
\label{inequalities}
\end{lemma}
\begin{proof}
  From \eqref{intform}, \eqref{r_alpha} we obtain the existence of a
  neighborhood ${\cal U}$ of $(\calE_c,V_h(\calE_c))$ and a constant $C_5>0$
  such that
\begin{equation}
|{\cal R}_3^{(h)}(\calE,V)| \leq
\sum_{|\alpha|=3}|r_\alpha^{(h)}(\calE,V)||V-V_h|^\alpha
\leq C_5|V-V_h|^3\ .
\label{calR3}
\end{equation}
We select ${\cal U}$ so that
no bifurcations of stationary points of $d^2$ occur and
there exists a
constant $C_6>0$ with $d_k(\calE)\geq C_6, k\neq h$ for each
$(\calE,V)\in{\cal U}$.  Relation \eqref{calR3} together with
\eqref{taylor_d2},\eqref{approxdist} yield \eqref{d2_est} for some $C_1, C_2>0$.

Moreover, we can find a neighborhood ${\cal W}$ of $\calE_c$ such that
there are no bifurcations of stationary points of $d^2$, and
the
inequalities \eqref{norm_equiv} hold for some $C_3, C_4>0$: in fact ${\cal
  A}_h$ depends continuously on $\calE$ and ${\cal A}_h(\calE_c)$ is positive
definite.
\end{proof}





\begin{figure}
\psfragscanon
\psfrag{calE}{${\cal E}$}\psfrag{cW}{${\cal W}$}\psfrag{S}{$\Sigma$}
\psfrag{V}{$V$}\psfrag{cV}{${\cal V}$}
\psfrag{(E,Vh)}{$\;\Gamma_h$}
\psfrag{(E,Vk)}{$\Gamma_k$}
\psfrag{}{}
\centerline{\epsfig{figure=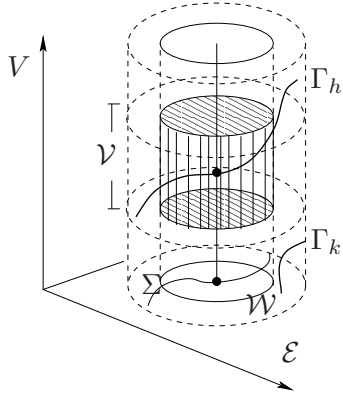,width=5cm}}
\psfragscanoff
\caption{Sketch for the selection of the neighborhood ${\cal U} = {\cal
    W}\times{\cal V}$. Here $\Gamma_j = \{(\calE,V_j(\calE)):d_j(\calE) = 0\}$
for $j=h,k$. In this case we restrict ${\cal W}$ to a smaller set (the inner
circle), as explained in the proof of Proposition~\ref{diff_bounded}.}
\label{chooseW}
\end{figure}

\begin{proposition}
  There exist $C>0$ and a neighborhood ${\cal W}$ of $\calE_c$
  such that
\[
\biggl|\frac{1}{d} - \frac{1}{\delta_h}\biggr| \le C
\hskip 1cm
 \forall\,(\calE, V) \in
({\cal W}\times\2toro)\setminus{\cal U}_\Sigma\,,
\]
where ${\cal U}_\Sigma = \{(\calE,V_h(\calE)): \calE\in\Sigma\}$ with $\Sigma =
\{\calE\in{\cal W}: d_h(\calE)=0\}$.
\label{diff_bounded}
\end{proposition}
\begin{proof}
  By Lemma~\ref{inequalities} we can choose two neighborhoods ${\cal W}$,
  ${\cal V}$ of $\calE_c$ and $V_h(\calE_c)$ respectively such that both
  \eqref{d2_est} and \eqref{norm_equiv} hold in ${\cal U} = {\cal
    W}\times{\cal V}$ .
We restrict ${\cal W}$, if necessary, so that there exists $C_7>0$ with
$d\geq C_7$ for each $(\calE, V)\in{\cal
  W}\times(\2toro\setminus{\cal V})$ (see Figure~\ref{chooseW}).
In ${\cal U}\setminus{\cal U}_\Sigma$ we have
\[
\biggl| {1\over d}-{1\over \delta_h}\biggr| = \frac{\vert
\delta_h^2 - d^2\vert}{\delta_h\, d
[\delta_h + d]} \leq
\frac{1}{\sqrt{C_1}[1+\sqrt{C_1}]}\frac{\vert \delta_h^2 -
d^2\vert}{\delta_h^3} \leq C
\]
for a constant $C>0$.
Using the boundedness of $1/d$, $1/\delta_h$ in ${\cal
  W}\times(\2toro\setminus{\cal V})$ we conclude the proof.
\end{proof}

Now we estimate the derivatives of the remainder function $1/d -
1/\delta_h$.

\begin{proposition}
There exist $C>0$ and a neighborhood ${\cal W}$ of $\calE_c$
such that, if $y_k$ is a component of Delaunay's elements $Y$, the estimate
\begin{equation}
  \left|\frac{\partial}{\partial y_k}\Bigl(\frac{1}{d}-
\frac{1}{\delta_h}\Bigr)\right|\leq
  \frac{C}{d_h+|V-V_h|}
\label{boundrem}
\end{equation}
holds for each $(\calE, V) \in ({\cal W}\times\2toro)\setminus {\cal U}_\Sigma$,
with ${\cal U}_\Sigma$ as in Proposition~\ref{diff_bounded}.  Therefore the
map
\begin{equation}
{\cal W}\setminus\Sigma \ni \calE \mapsto
\int_{\2toro}\frac{\partial}{\partial y_k}\Bigl(\frac{1}{d}-\frac{1}{\delta_h}\Bigr)\,d\ell
d\ell'\,,
\label{int_rem}
\end{equation}
where $\Sigma = \{\calE\in{\cal W}: d_h(\calE)=0\}$, can be extended continuously to
the whole set ${\cal W}$.
\label{extend_rem}
\end{proposition}
\begin{proof}
  In the following we denote by $C_j, j=8\ldots 14$ some positive
  constants. We write
\[
\frac{\partial}{\partial y_k}\Bigl(\frac{1}{d}-\frac{1}{\delta_h}\Bigr) =
\frac{1}{2}\Bigl(\frac{1}{\delta_h^3} - \frac{1}{d^3}\Bigr)\frac{\partial \delta_h^2}{\partial y_k} -
\frac{1}{2d^3}\frac{\partial {\cal R}_3^{(h)}}{\partial y_k}\,,
\]
and give an estimate for the two terms at the right hand side.  We choose a
neighborhood ${\cal U} = {\cal W}\times{\cal V}$ of $(\calE_c,V_h(\calE_c))$
as in Proposition~\ref{diff_bounded} so that, using \eqref{d2_est},
\eqref{norm_equiv} and the boundedness of the
remainder function, we have
\begin{eqnarray*}
\biggl|\frac{1}{\delta_h^3} - \frac{1}{d^3}\biggr| &=&
\left|\frac{1}{\delta_h} -
\frac{1}{d}\right|\Bigl(\frac{1}{\delta_h^2} +\frac{1}{\delta_h d} +
\frac{1}{d^2}\Bigr) \leq
\frac{C_8}{d_h^2 + |V-V_h|^2}
\end{eqnarray*}
in ${\cal U}_0={\cal U}\setminus{\cal U}_\Sigma$.
Moreover in ${\cal U}_0$ we have
\[
\biggl|\frac{\partial \delta_h^2}{\partial y_k}\biggr| \leq
\biggl|\frac{\partial d_h^2}{\partial y_k}\biggr| + C_9|V-V_h| \leq C_{10}(d_h
+ |V-V_h|)\,,
\]
since
\begin{equation}
  \frac{\partial \delta_h^2}{\partial y_k} = \frac{\partial d_h^2}{\partial y_k}
  - 2\frac{\partial V_h}{\partial y_k}\cdot{\cal A}_h(V-V_h) + (V-V_h)\cdot\frac{\partial {\cal A}_h}{\partial y_k}(V-V_h)\,,
\label{ddeltah2dyk}
\end{equation}
and the derivatives
\[
\frac{\partial V_h}{\partial  y_k}(\calE) = -[{\cal
  H}_h(\calE)]^{-1}\frac{\partial}{\partial  y_k}\nabla_V
d^2(\calE,V_h(\calE))
\] are uniformly bounded for $\calE\in{\cal W}$ since bifurcations do not occur.

Hence the relation
\begin{equation}
\biggl|\Bigl(\frac{1}{\delta_h^3} - \frac{1}{d^3}\Bigr)\frac{\partial
    \delta_h^2}{\partial y_k}\biggr| \leq
C_{11}\frac{d_h+|V-V_h|}{d_h^2+|V-V_h|^2} \leq \frac{2C_{11}}{d_h+|V-V_h|}
\label{primo_pezzo}
\end{equation}
holds in ${\cal U}_0$,
with $C_{11}=C_8C_{10}$. We also have
\begin{equation}
\biggl|\frac{\partial{\cal R}_3^{(h)}}{\partial y_k}\biggr|
\leq C_{13}|V-V_h|^2\,,
\label{stima_derR3}
\end{equation}
for $(\calE,V)\in{\cal U}_0$, in fact
\begin{equation}
\sup_{{\cal U}_0}|r_\alpha^{(h)}| < +\infty\,, \qquad \sup_{{\cal U}_0}\biggl|\frac{\partial
  r_\alpha^{(h)}}{\partial y_k}\biggr| < +\infty\,,
\label{r_alpha_bounds}
\end{equation}
for each $\alpha=(\alpha_1,\alpha_2)$ with $|\alpha|=3$.
%
Using again \eqref{d2_est}, \eqref{norm_equiv} we obtain
\begin{equation}
\biggl|\frac{1}{d^3}\frac{\partial {\cal R}_3^{(h)}}{\partial  y_k}\biggr|\leq
\frac{C_{14}}{d_h+|V-V_h|}\ .
\label{secondo_pezzo}
\end{equation}
From \eqref{primo_pezzo}, \eqref{secondo_pezzo}
we obtain \eqref{boundrem}
and the assert of the proposition follows using
the boundedness of $\frac{\partial}{\partial
  y_k}\bigl({1}/{d}\bigr), \frac{\partial}{\partial
  y_k}\bigl({1}/{\delta_h}\bigr)$ in ${\cal
  W}\times(\2toro\setminus{\cal V})$.
\end{proof}

From \eqref{boundrem} in Proposition~\ref{extend_rem} the average over
$\2toro$ of the derivatives of $1/d - 1/\delta_h$ in \eqref{int_rem} is finite
for each $\calE$ in ${\cal W}$, and can be computed by exchanging the integral
and differential operators: therefore the average of the remainder function is
continuously differentiable in ${\cal W}$.

On the other hand, the average over $\2toro$ of the derivatives with respect
to $y_k$ of $1/\delta_h$ are non-convergent integrals for $\calE\in\Sigma$:
for this reason the averaged vector field in \eqref{averrhs} is not defined at
orbit crossings.  Next we show, exchanging again the integral
and differential operators, that the average of these derivatives admit two
analytic extensions to the whole ${\cal W}$ from both sides of the singular
set $\Sigma$.

\noindent For this purpose, given a neighborhood ${\cal W}$ of $\calE_c$, we set
\[
{\cal W}^+ = {\cal W}\cap\{\tilde{d}_h>0\}\,,
\hskip 0.7cm
{\cal W}^- = {\cal W}\cap\{\tilde{d}_h<0\}\,,
\]
with $\tilde{d}_h$ given by \eqref{choosesign}.

\begin{proposition}
  There exists a neighborhood ${\cal W}$ of $\calE_c$ such that
the maps
\[ {\cal W}^+ \ni \calE \mapsto \frac{\partial}{\partial y_k}\int_{\2toro}
\frac{1}{\delta_h}\,d\ell\,d\ell'\,, \hskip 0.7cm {\cal W}^- \ni \calE \mapsto
\frac{\partial}{\partial y_k}\int_{\2toro}
\frac{1}{\delta_h}\,d\ell\,d\ell'\,,
\]
where $y_k$ is a component of Delaunay's elements $Y$, can be extended
to two different analytic maps ${\cal G}_{h,k}^+, {\cal G}_{h,k}^-$, defined
in ${\cal W}$.
\label{extend_approx}
\end{proposition}
\begin{proof}

  We choose ${\cal W}$ as in Proposition~\ref{extend_rem} and, if necessary,
  we restrict this neighborhood by requiring that
  $\tau_1^{(h)}\times\tau_2^{(h)}\neq 0$ in ${\cal W}$, so that
  $\tilde{d}_h|_{{\cal W}}$ is analytic.  To investigate the behavior close to
  the singularity, for each $\calE\in{\cal W}$, we can restrict the integrals
  to the domain
\begin{equation}
{\cal D} = {\cal D}(V_h,r) = \{V\in\2toro: (V-V_h)\cdot {\cal A}_h(V-V_h)
\leq r^2\}\,,
\label{domain}
\end{equation}
for a suitable $r>0$.
By using the coordinate change $\xi = {\cal A}_h^{1/2}(V-V_h)$ and then
polar coordinates $(\rho,\theta)$, defined by
$(\rho\cos\theta,\rho\sin\theta)=\xi$, we have
\begin{IEEEeqnarray*}{rCl}
\int_{{\cal D}}\frac{1}{\delta_h}\,d\ell d\ell' &=&
  \frac{1}{\sqrt{\det{\cal A}_h}}\int_{{\cal B}}
  \frac{1}{\sqrt{d_h^2 + |\xi|^2}}\,d\xi \\
&=& \frac{2\pi}{\sqrt{\det{\cal A}_h}}
  \int_0^r\frac{\rho}{\sqrt{d_h^2 + \rho^2}}\,d\rho =
 \frac{2\pi}{\sqrt{\det{\cal A}_h}}(\sqrt{d_h^2+r^2} - d_h) \,,
\end{IEEEeqnarray*}
with ${\cal B} = \{\xi\in\R^2:|\xi|\leq r\}$.  The term $-2\pi
d_h/\sqrt{\det{\cal A}_h}$ is not differentiable at $\calE =
\calE_c\in\Sigma$.  We set
\[
{\cal F}_{h,k} = \frac{\partial}{\partial
  y_k}\Bigl(\frac{2\pi}{\sqrt{\det{\cal A}_h}}\Bigr) \sqrt{d_h^2 + r^2} +
\frac{2\pi}{\sqrt{\det{\cal A}_h}} \frac{\tilde{d}_h}{\sqrt{d_h^2 + r^2}}
\frac{\partial \tilde{d}_h}{\partial y_k}
\]
with $\tilde{d}_h$ as in \eqref{choosesign}, and define on ${\cal W}$
the two analytic maps
\begin{equation}
{\cal G}_{h,k}^\pm = {\cal F}_{h,k} \mp \frac{\partial}{\partial
  y_k}\Bigl(\frac{2\pi}{\sqrt{\det{\cal A}_h}}\Bigr)\tilde{d}_h \mp
\frac{2\pi}{\sqrt{\det{\cal A}_h}} \frac{\partial \tilde{d}_h}{\partial
  y_k}
+ \frac{\partial}{\partial y_k}\int_{\2toro\setminus{\cal D}}
\frac{1}{\delta_h}\,d\ell\,d\ell' \ .
\label{derdh}
\end{equation}
We observe that ${\cal G}_{h,k}^+$ (resp. ${\cal G}_{h,k}^-$) corresponds to the
derivative of $\int_{\2toro} 1/\delta_h\,d\ell\,d\ell'$ with respect to $y_k$
on ${\cal W}^+$ (resp. ${\cal W}^-$).
\end{proof}

Now we state the main result.
\begin{theorem}
  The averages over $\2toro$ of the derivatives of $R$ with respect to
  Delaunay's elements $y_k$ can be extended to two Lipschitz--continuous maps
  $\bigl(\frac{\ov{\partial R}}{\partial y_k}\bigr)_h^\pm$ on a neighborhood
  ${\cal W}$ of $\calE_c$.
  These maps, restricted to ${\cal W}^+$, ${\cal W}^-$ respectively,
  correspond to $\frac{\ov{\partial R}}{\partial y_k}$.  Moreover the
  following relations hold:
\begin{eqnarray}
{\rm Diff}_h \left(\frac{\ov{\partial R}}{\partial y_k}\right)
&\stackrel{def}{=}&
\Bigl(\frac{\ov{\partial R}}{\partial y_k}\Bigr)_h^- -
\Bigl(\frac{\ov{\partial R}}{\partial y_k}\Bigr)_h^+ = \nonumber \\
&=& \frac{\mu
  k^2}{\pi}\biggl[
\frac{\partial}{\partial y_k}\biggl(\frac{1}{\sqrt{\det(\calA_h)}}\biggr)\tilde{d}_h +
\frac{1}{\sqrt{\det(\calA_h)}} \frac{\partial \tilde{d}_h}{\partial
  y_k}
\biggr]\,,
\label{diff}
\end{eqnarray}
with the derivatives of $\tilde{d}_h$ given by \eqref{derdhtilde}.
\label{der_jumps}
\end{theorem}

\begin{proof}
Using the results of Propositions~\ref{extend_rem}, \ref{extend_approx} we
define the extended maps by
\[
\Bigl(\frac{\ov{\partial R}}{\partial  y_k}\Bigr)_h^\pm =
\frac{\mu k^2}{(2\pi)^2}\biggl[
\int_{\2toro} \frac{\partial}{\partial  y_k} \left(\frac{1}{d} -
\frac{1}{\delta_h}\right)\,d\ell\, d\ell' + {\cal G}_{h,k}^\pm\biggr]\,,
\]
with ${\cal G}_{h,k}^\pm$ given by \eqref{derdh}.
We show that the maps $\calE
\mapsto \bigl(\frac{\ov{\partial R}}{\partial y_k}\bigr)_h^\pm(\calE)$ are
Lipschitz--continuous extensions to ${\cal W}$ of $\frac{\ov{\partial R}}{\partial y_k}$.
The maps ${\cal G}_{h,k}^\pm$ are analytic in ${\cal W}$, thus we only
have to study the integrals $\int_{\2toro}\frac{\partial}{\partial y_k}({1/d} -
{1/\delta_h})\;d\ell d\ell'$. From Proposition~\ref{extend_rem} we know
that these maps are continuous.

We only need to investigate the behavior close to the singularity, therefore
we restrict these integrals to the domain ${\cal D}$ introduced in
\eqref{domain}.  We prove that the maps
\[
{\cal W}\setminus\Sigma \ni \calE \mapsto \frac{\partial}{\partial y_j}\int_{\cal D}\frac{\partial}{\partial
  y_k}\frac{1}{\delta_h}\,d\ell d\ell'\,, \hskip 1cm
{\cal W}\setminus\Sigma \ni \calE \mapsto \frac{\partial}{\partial
  y_j}\int_{\cal D}\frac{\partial}{\partial y_k}\frac{1}{d}\,d\ell d\ell'\,,
\]
with $j=1\ldots 4$, are bounded.
First observe that the derivatives
\begin{equation}
\frac{\partial}{\partial y_j}\int_{\cal D}\frac{\partial}{\partial
  y_k}\frac{1}{\delta_h}\,d\ell d\ell'  =
\int_{\cal D}\biggl(\frac{3}{4}\frac{1}{\delta_h^5}
\frac{\partial\delta_h^2}{\partial y_j}\frac{\partial\delta_h^2}{\partial y_k}
-\frac{1}{2}\frac{1}{\delta_h^3} \frac{\partial^2\delta_h^2}{\partial
  y_j\partial y_k}\biggr)\,d\ell d\ell'
\label{secderdelta}
\end{equation}
are bounded in ${\cal W}\setminus \Sigma$, otherwise we could not find the
analytic extensions ${\cal G}_{h,k}^+, {\cal G}_{h,k}^-$ introduced in
Proposition~\ref{extend_approx}.\footnote{Actually we can prove that
\begin{eqnarray*}
\frac{3}{4}\int_{\cal D} \frac{1}{\delta_h^5}\frac{\partial\delta_h^2}{\partial y_j}\frac{\partial\delta_h^2}{\partial y_k}\,d\ell d\ell' =
\mathfrak{T}_{j,k}^{(h)} + \mathfrak{U}_{j,k}^{(h)}\,,
\hskip 1cm
\frac{1}{2}\int_{\cal D}\frac{1}{\delta_h^3} \frac{\partial^2\delta_h^2}{\partial
  y_j\partial y_k}\,d\ell d\ell' = \mathfrak{T}_{j,k}^{(h)} + \mathfrak{V}_{j,k}^{(h)}\,,
\end{eqnarray*}
where $\mathfrak{U}_{j,k}^{(h)},\mathfrak{V}_{j,k}^{(h)}$ are bounded in ${\cal
  W}\setminus\Sigma$, and
\[
\mathfrak{T}_{j,k}^{(h)} = \frac{2\pi}{d_h\sqrt{\det{\cal A}_h}}\left(\frac{\partial
    d_h}{\partial \calE_j}\frac{\partial d_h}{\partial
    \calE_k} + \frac{\partial V_h}{\partial \calE_j}\cdot{\cal
    A}_h\frac{\partial V_h}{\partial \calE_k}\right)
\]
is unbounded but cancels out in the difference.}
Then we show that the maps
\begin{eqnarray}
\frac{\partial}{\partial y_j}\int_{\cal D}\frac{\partial}{\partial
  y_k}\frac{1}{d}\,d\ell d\ell' &=&
\int_{\cal D}\biggl(\frac{3}{4}\frac{1}{d^5}
\frac{\partial d^2}{\partial y_j}\frac{\partial d^2}{\partial y_k}
-\frac{1}{2}\frac{1}{d^3} \frac{\partial^2 d^2}{\partial
  y_j\partial y_k}\biggr)\,d\ell d\ell'
\label{secder1/d}
\end{eqnarray}
are bounded in ${\cal W}\setminus\Sigma$.
Using \eqref{taylor_d2}, \eqref{approxdist} we write the integrand function in the right hand
side of \eqref{secder1/d} as the sum of
\begin{equation}
 \frac{3}{4}\frac{1}{\delta_h^5} \frac{\partial \delta_h^2}{\partial
    y_j}\frac{\partial \delta_h^2}{\partial y_k}\,,
\hskip 0.5cm
 -\frac{1}{2}\frac{1}{\delta_h^3} \frac{\partial^2 \delta_h^2}{\partial
    y_j\partial y_k}
\label{already_done}
\end{equation}
and of terms of the following kind:
\begin{equation}
 \frac{3}{4}\frac{1}{\delta_h^5} \frac{\partial {\cal R}_3^{(h)}}{\partial
   y_j}\left[\frac{\partial {\cal R}_3^{(h)}}{\partial y_k}\bigl(1+{\cal
     P}_5^{(h)}\bigr) + {\cal P}_5^{(h)}\frac{\partial \delta_h^2}{\partial y_k}\right]\,,
 \hskip 0.5cm
 -\frac{1}{2}\frac{{\cal P}_3^{(h)}}{\delta_h^3} \frac{\partial^2 {\cal
      R}_3^{(h)}}{\partial y_j\partial y_k}\,,
\label{easy_bound}
\end{equation}
\begin{equation}
 \frac{3}{4}\frac{1}{\delta_h^5} \frac{\partial {\cal R}_3^{(h)}}{\partial
   y_j}\frac{\partial \delta_h^2}{\partial y_k}\,,
\hskip 0.5cm -\frac{1}{2}\frac{1}{\delta_h^3} \frac{\partial^2 {\cal
    R}_3^{(h)}}{\partial y_j\partial y_k}\,,
\label{middle_bound}
\end{equation}
\begin{equation}
\frac{3}{4}\frac{{\cal P}_5^{(h)}}{\delta_h^5} \frac{\partial
  \delta_h^2}{\partial y_j}\frac{\partial \delta_h^2}{\partial y_k}\,,
\hskip 0.5cm -\frac{1}{2}\frac{{\cal P}_3^{(h)}}{\delta_h^3}
\frac{\partial^2 \delta_h^2}{\partial y_j\partial y_k}\ .
\label{difficult_bound}
\end{equation}

\noindent The integrals over ${\cal D}$ of the terms in \eqref{already_done}
are not bounded in ${\cal W}\setminus\Sigma$, but their sum is bounded and
corresponds to \eqref{secderdelta}.
In the following we denote by $C_j, j=15\ldots 34$ some positive constants.
Moreover we use the relation $d^2 = \delta_h^2 + {\cal R}_3^{(h)}$ and
the developments
\[
\frac{1}{d^s} =
\frac{1}{(\delta_h^2+{\cal R}_3^{(h)})^{s/2}} = \frac{1}{\delta_h^s}\left[1
 +{\cal P}_s^{(h)}\right]
\hskip 1cm (s=3,5)
\]
with
\[
{\cal P}_s^{(h)} = {\cal P}_s^{(h)}(\calE,V) = \sum_{|\beta|=1}
p_{\beta,s}^{(h)}({\cal E},V)(V-V_h)^\beta\,,
\]
\begin{equation}
p_{\beta,s}^{(h)}({\cal E},V) = \int_0^1 D^\beta\biggl[\Bigl(1+\frac{{\cal
    R}_3^{(h)}}{\delta_h^2}\Bigr)^{-s/2}\biggr](\calE,V_h+t(V-V_h))\,dt\ .
\label{pbetas}
\end{equation}
By developing \eqref{pbetas} we obtain
\begin{eqnarray*}
  p_{\beta,s}^{(h)}({\cal E},V)
  &=&
  -\frac{s}{2}\int_0^1 \biggl[\biggl(1+\frac{{\cal
      R}_3^{(h)}}{\delta_h^2}\biggr)^{-\frac{s}{2}-1}D^\beta
  \biggl(\frac{{\cal R}_3^{(h)}}{\delta_h^2}\biggr)\biggr](\calE,V_h+t(V-V_h))\,dt
 \\
&=&
-\frac{s}{2}\int_0^1
D^\beta\biggl(\frac{{\cal R}_3^{(h)}}{\delta_h^2}\biggr)(\calE,V_h+t(V-V_h))\,dt +
\mathfrak{R}_s^{(h)}(\calE, V)\,,
\end{eqnarray*}
with $|\mathfrak{R}_s^{(h)}(\calE, V)|\leq C_{15}|V-V_h|, s=3,5$.
Moreover, we have
\begin{equation}
D^\beta\left(\frac{{\cal R}_3^{(h)}}{\delta_h^2}\right) = \frac{D^\beta{\cal
    R}_3^{(h)}}{\delta_h^2} - \frac{{\cal
    R}_3^{(h)}}{\delta_h^4}D^\beta\delta_h^2
\ .
\label{DbetaR3sudelta2}
\end{equation}
We can estimate the terms in \eqref{DbetaR3sudelta2} as follows:
\[
D^\beta{\cal R}_3^{(h)} = \sum_{|\alpha|=3} \left[D^\beta r_\alpha^{(h)}
  (V-V_h)^\alpha + r_\alpha^{(h)} D^\beta (V-V_h)^\alpha\right]
\]
where
\[
|D^\beta r_\alpha^{(h)}|\leq C_{16}\,,
\hskip 0.5cm
|D^\beta (V-V_h)^\alpha|\leq C_{17}|V-V_h|^2\,,
\]
so that
\[
|D^\beta {\cal R}_3^{(h)}|\leq C_{18}|V-V_h|^2
\]
Moreover
\[
D^\beta\delta_h^2 = 2D^\beta(V-V_h)\cdot{\cal A}_h(V-V_h)
\]
so that
\[
|D^\beta\delta_h| \leq C_{19}|V-V_h|\ .
\]
We conclude that
\[
\left|D^\beta\left(\frac{{\cal R}_3^{(h)}}{\delta_h^2}\right)\right|\leq C_{20}\,,
\ \mbox{ so that } \ |p_{\beta,s}^{(h)}(\calE,V)|\leq C_{21}\,,
\]
and we obtain the estimate
\begin{equation}
|{\cal P}_s^{(h)}(\calE,V)| \leq C_{22}|V-V_h|
\label{stima_Ps}
\end{equation}
for $(\calE,V)\in{\cal U}_0$.
Using \eqref{ddeltah2dyk}, \eqref{stima_derR3}, \eqref{stima_Ps}
and the estimate
\[
\biggl|\frac{\partial^2 {\cal R}_3^{(h)}}{\partial
    y_j\partial y_k}\biggr| \leq C_{23}|V-V_h|\,,
\]
that follows from the boundedness of
\[
r_\alpha^{(h)},\quad \frac{\partial r_\alpha^{(h)}}{\partial y_k},\quad
\frac{\partial^2 r_\alpha^{(h)}}{\partial y_j\partial y_k},\quad
\frac{\partial V_h}{\partial y_k},\quad \frac{\partial^2 V_h}{\partial
  y_j\partial y_k}\,,
\]
we can bound both terms in \eqref{easy_bound} by $C_{24}/|V-V_h|$, which
has finite integral over ${\cal D}$.\footnote{The boundedness of $\frac{\partial^2 V_h}{\partial
  y_j\partial y_k}$ on ${\cal W}$ follows by differentiating
with respect to $y_j$ the relation
\[ {\cal H}_h(\calE)\frac{\partial V_h}{\partial y_k}(\calE) =
-\frac{\partial}{\partial y_k}\nabla_V d^2(\calE,V_h(\calE))\ .
\]
}

\noindent To estimate the integrals of the terms in
\eqref{middle_bound} we observe that
\[
\frac{\partial{\cal R}_3^{(h)}}{\partial y_j} =
\sum_{|\alpha|=3}r_{\alpha,0}^{(h)}\frac{\partial(V-V_h)^\alpha}{\partial
  y_j} + \mathfrak{S}_3^{(h)}\,,
\quad
\frac{\partial^2 {\cal R}_3^{(h)}}{\partial y_j\partial y_k} =
\sum_{|\alpha|=3}r_{\alpha,0}^{(h)}\frac{\partial^2 (V-V_h)^\alpha}{\partial
  y_j\partial y_k} + \mathfrak{S}_2^{(h)}\,,
\]
with
\[
r_{\alpha,0}^{(h)} = r_{\alpha,0}^{(h)}(\calE) = r_\alpha^{(h)}(\calE,V_h(\calE))\,,
\qquad
|\mathfrak{S}_i^{(h)}|\leq C_{25}|V-V_h|^i\qquad (i=2,3)\ .
\]
Then, using \eqref{ddeltah2dyk} and writing $dV$ for $d\ell d\ell'$,
we have
\begin{eqnarray}
&&\hskip -0.8cm\biggl|\int_{\cal D}\frac{1}{\delta_h^5}\frac{\partial {\cal R}_3^{(h)}}{\partial
    y_j}\frac{\partial \delta_h^2}{\partial y_k}\,dV\biggr| \leq
\biggl|\frac{\partial d_h^2}{\partial y_k}\biggr|\biggl(
\sum_{|\alpha|=3}\bigl|r_{\alpha,0}^{(h)}\bigr|\int_{\cal D}\left|\frac{1}{\delta_h^5}
\frac{\partial(V-V_h)^\alpha}{\partial y_j}\right|\,dV + \int_{\cal D}\frac{|\mathfrak{S}_3^{(h)}|}{\delta_h^5}\,dV\biggl)
\nonumber\\
&&\label{prima}\\
&&\hskip -0.8cm +2\sum_{|\alpha|=3}\bigl|r_{\alpha,0}^{(h)}\bigr|\left|\int_{\cal
    D}\frac{1}{\delta_h^5} \frac{\partial(V-V_h)^\alpha}{\partial
    y_j}\left[\frac{\partial V_h}{\partial y_k}\cdot{\cal A}_h(V-V_h)
  \right]\,dV\right| + C_{26}\int_{\cal D}\frac{|V-V_h|^4}{\delta_h^5}\,dV\nonumber
\end{eqnarray}
and
\begin{equation}
\left|\int_{\cal D}\frac{1}{\delta_h^3}\frac{\partial^2 {\cal R}_3^{(h)}}{\partial
    y_j\partial y_k}\,dV\right| \leq
\sum_{|\alpha|=3}\bigl|r_{\alpha,0}^{(h)}\bigr|\left|\int_{\cal D}\frac{1}{\delta_h^3}
\frac{\partial^2(V-V_h)^\alpha}{\partial
    y_j\partial y_k}
\,dV\right|
+ C_{27}\int_{\cal D}\frac{|V-V_h|^2}{\delta_h^3}\,dV\ .
\label{seconda}
\end{equation}
Passing to polar coordinates $(\rho,\theta)$, defined by
$(\rho\cos\theta,\rho\sin\theta) = {\cal A}_h^{1/2}(V-V_h)$,
we find that
\[
\biggl|\frac{\partial d_h^2}{\partial y_k}\biggr|
\int_{\cal D}\biggl|\frac{1}{\delta_h^5}
\frac{\partial(V-V_h)^\alpha}{\partial y_j}\biggr| \,dV \leq C_{28}\,,
\hskip 0.5cm
\biggl|\frac{\partial d_h^2}{\partial y_k}\biggr|\int_{\cal
  D}\frac{|\mathfrak{S}_3^{(h)}|}{\delta_h^5}\,dV \leq C_{29}
\]
for each $\calE\in{\cal W}\setminus\Sigma$ and $\alpha$ with
$|\alpha|=3$, in fact
\begin{equation}
\int_0^r\frac{\rho^i}{(d_h^2+\rho^2)^{5/2}}\,d\rho \leq \frac{C_{30}}{d_h}
\hskip 1cm
(i=3,4)\ .
\label{integrals}
\end{equation}
Moreover, passing to polar coordinates $(\rho,\theta)$, we have
\begin{eqnarray}
&&\int_{\cal D}\frac{1}{\delta_h^5}
\frac{\partial(V-V_h)^\alpha}{\partial y_j}
\left[\frac{\partial V_h}{\partial y_k}\cdot{\cal A}_h(V-V_h)\right]\,dV
=\nonumber \\
&=&
    \int_0^r \frac{\rho^4}{(d_h^2+\rho^2)^{5/2}}\,d\rho
\sum_{|\gamma|=3}c_\gamma
\int_0^{2\pi}(\cos\theta)^{\gamma_1}(\sin\theta)^{\gamma_2}\,d\theta  = 0
\label{sndbound}
\end{eqnarray}
for some functions $c_\gamma:{\cal W}\setminus\Sigma\to\R$,
$\gamma=(\gamma_1,\gamma_2)$. Thus the integrals in \eqref{prima} are
uniformly bounded in ${\cal W}\setminus\Sigma$.
In \eqref{sndbound} we have used
\begin{equation}
\int_0^{2\pi}(\cos\theta)^{\gamma_1}(\sin\theta)^{\gamma_2}
\,d\theta = 0
\label{vanish}
\end{equation}
for each $\gamma$, with odd $|\gamma| = \gamma_1+\gamma_2$ .
Finally, using again \eqref{vanish}, we obtain
\begin{eqnarray*}
\biggl|\int_{\cal D}\frac{1}{\delta_h^3}
  \frac{\partial^2(V-V_h)^\alpha}{\partial y_j\partial y_k} \,dV\biggr| &\leq&
  \biggl|\sum_{|\gamma|=1}b_\gamma
\int_0^{2\pi}(\cos\theta)^{\gamma_1}(\sin\theta)^{\gamma_2}\,d\theta\biggr|
    \int_0^r \frac{\rho^2}{(d_h^2+\rho^2)^{3/2}}\,d\rho \\
&+& C_{31}\int_{\cal D}\frac{1}{|V-V_h|}\,dV = C_{31}\int_{\cal D}\frac{1}{|V-V_h|}\,dV
\end{eqnarray*}
for some functions $b_\gamma:{\cal W}\setminus\Sigma\to\R$.
Hence also the integrals in \eqref{seconda} are uniformly bounded in
${\cal W}\setminus\Sigma$.

\noindent To estimate the integrals of the terms in
\eqref{difficult_bound} we make the following decomposition:
\[
p_{\beta,s}^{(h)} = q_{\beta,s}^{(h)} + w_{\beta,s}^{(h)}\,,
\]
with

\begin{align*}
&q_{\beta,s}^{(h)} \\
=&
-\frac{s}{2}\sum_{|\alpha|=3}r_{\alpha,0}^{(h)} \int_0^1
\left[\frac{1}{\delta_h^2}\biggl(D^\beta(V-V_h)^\alpha
  -\frac{D^\beta\delta_h^2}{\delta_h^2}(V-V_h)^\alpha
  \biggr)\right](\calE,V_h+t(V-V_h))\,dt
\end{align*}
and $|w_{\beta,s}^{(h)}| \leq C_{32}|V-V_h|$.
For the first term in \eqref{difficult_bound} we obtain
\begin{eqnarray}
&&\biggl|\int_{\cal D}\frac{{\cal P}_5^{(h)}}{\delta_h^5} \frac{\partial
      \delta_h^2}{\partial y_j}\frac{\partial \delta_h^2}{\partial
      y_k}\,dV\biggr| \leq
 \biggl|\frac{\partial
      d_h^2}{\partial y_j}\frac{\partial d_h^2}{\partial y_k}
 \int_{\cal
   D}\frac{1}{\delta_h^5}\sum_{|\beta|=1}q_{\beta,5}^{(h)}(V-V_h)^\beta\,dV\biggr|
\nonumber\\
&& \label{ultimastima}\\
&+& 4\biggl|\int_{\cal D}\frac{1}{\delta_h^5}\sum_{|\beta|=1}q_{\beta,5}^{(h)}(V-V_h)^\beta
      \left[\frac{\partial V_h}{\partial
        y_j}\cdot{\cal A}_h(V-V_h)\right]\left[\frac{\partial V_h}{\partial
        y_k}\cdot{\cal A}_h(V-V_h)\right]\,dV\biggr|\nonumber \\
\nonumber
&+& C_{33}\nonumber
\end{eqnarray}
where we have used polar coordinates and the inequalities \eqref{integrals}.
The two integrals at the right hand side of \eqref{ultimastima} vanish: in
fact using Fubini-Tonelli's theorem and passing to polar coordinates
$(\rho,\theta)$, by relations \eqref{vanish} we obtain
\[
\begin{array}{l}
\displaystyle\int_{\cal
  D}\frac{1}{\delta_h^5}\sum_{|\beta|=1}q_{\beta,5}^{(h)}(V-V_h)^\beta\,dV
= \\
=\displaystyle\sum_{|\beta|=1}
\sum_{|\gamma|\in\{3,5\}}\int_0^1\int_0^r \phi_{\beta,\gamma}(\rho,t)\,d\rho\,dt
\int_0^{2\pi}
(\cos\theta)^{\gamma_1} (\sin\theta)^{\gamma_2}
\,d\theta = 0
\end{array}
\]
for some functions $\phi_{\beta,\gamma}:\R^+\times\R\to\R$. The computation
for the other integral is analogous.

\noindent The second term in \eqref{difficult_bound} is estimated in a similar
way:
\begin{eqnarray*}
\small\hskip -1cm &&\biggl|\int_{\cal D}\frac{{\cal P}_3^{(h)}}{\delta_h^3}
\frac{\partial^2 \delta_h^2}{\partial y_j\partial y_k}\,dV\biggr| \\
\leq
&&\biggl|\biggl(\frac{\partial^2 d_h^2}{\partial y_j\partial y_k}
+ 2\frac{\partial V_h}{\partial y_j}\cdot{\cal A}_h\frac{\partial V_h}{\partial y_k}\biggr)  \int_{\cal D}\frac{1}{\delta_h^3}
\sum_{|\beta|=1}q_{\beta,3}^{(h)}(V-V_h)^\beta\,dV\biggr| + C_{34}\,,
\end{eqnarray*}
and the integral at the right hand side vanishes as well.

We conclude the proof observing that, using \eqref{derdh} and the theorem of
differentiation under the integral sign, the derivatives
$\bigl(\frac{\ov{\partial R}}{\partial y_k}\bigr)_h^+$,
$\bigl(\frac{\ov{\partial R}}{\partial y_k}\bigr)_h^-$, restricted to ${\cal
  W}^+$, ${\cal W}^-$ respectively, correspond to $\frac{\ov{\partial
    R}}{\partial y_k}$, and their difference in ${\cal W}$ is given by \eqref{diff}.
\end{proof}

\begin{remark}
  If $\calE_c$ is an orbit configuration with two crossings, assuming that
  $d_h(\calE_c) = 0$ for $h=1,2$, we can extract the singularity by
  considering the approximated distances $\delta_1, \delta_2$ and the
  remainder function $1/d - 1/\delta_1 - 1/\delta_2$.
\end{remark}

\section{Generalized solutions}
\label{s:gensol}

We show that generically we can uniquely extend the solutions of
\eqref{averrhs} beyond the crossing singularity $d_{min}=0$.  This is obtained
by patching together classical solutions defined in the domains ${\cal W}^+$
with solutions defined in ${\cal W}^-$, or {\em vice versa}.

Let $a>0$ be a value for the semimajor axis of the asteroid and $\ov Y:I\to
\R^4$ be a continuous function defined in an open interval $I\subset\R$,
representing a possible evolution of the asteroid orbital elements
$Y=(G,Z,g,z)$.  We introduce
\begin{equation}
{\ov\calE}(t) = ({\ov E}(t), {\ov E}'(t))\,,
\label{pair}
\end{equation}
with
\begin{equation}
{\ov E}(t) = (k\sqrt{a},{\ov Y}(t))\,,
\label{astorb}
\end{equation}
where $k$ is Gauss' constant and ${\ov E}'$ is a known function of time
representing the evolution of the Earth.\footnote{In the case of one
  perturbing planet ${\ov E}'(t)$ is constant and represents the trajectory of
  a solution of the 2-body problem. If we consider more than one perturbing
  planet then ${\ov E}'(t)$ changes with time due to the planetary
  perturbations.}

Let $T({\ov Y})$ be the set of times $t_c\in I$ such that
$d_{min}({\ov\calE}(t_c)) = 0$, and assume that each $t_c$ is isolated, so
that we can represent the set
\[
I \setminus T({\ov Y}) = \displaystyle\sqcup_{j\in {\cal N}} I_j
\]
as disjoint union of open intervals $I_j$, with ${\cal N}$ a countable
(possibly finite) set.
\begin{definition}
  We say that $\ov Y$ is a generalized solution of \eqref{averrhs} if its
  restriction to each $I_j, j\in {\cal N}$ is a classical solution of \eqref{averrhs}
  and, for each $t_c\in T({\ov Y})$, there exist finite values of
\[
\lim_{t\to t_c^+}\dot{\ov Y}(t)\,,\hskip 0.5cm
\lim_{t\to t_c^-}\dot{\ov Y}(t)\ .
\]
\end{definition}

Choose $Y_0 \in\R^4$ and a time $t_0$ such that $d_{min}(\calE_0)>0$, with
$\calE_0 = (E_0,E_0')$, $E_0=(k\sqrt{a},Y_0)$, $E_0' = E'(t_0)$.
We show how we can construct a generalized solution of the Cauchy problem
\begin{equation}
  \dot{\ov Y}  = -\mathbb{J}_2\,\ov{ \nabla_Y R}\,,
\hskip 1cm
\ov Y(t_0) = Y_0\ .
\label{cauchy}
\end{equation}
Let $\ov Y(t)$ be the maximal classical solution of \eqref{cauchy}, defined in
the maximal interval $J$. Assume that $t_c=\sup J<+\infty$, and $\lim_{t\to
  t_c^-}\barcalE(t) = \calE_c$, with $\calE_c$ a non-degenerate crossing
configuration such that $d_{min}(\calE_c) = d_h(\calE_c) = 0$ for some $h$.  Let ${\cal W}$,
${\cal W}^\pm$ be chosen as in Theorem~\ref{der_jumps}.  Suppose that
there exists $\tau\in(t_0,t_c)$ such that
${\ov \calE}(t)\in {\cal W}^+$ for
$t\in(\tau,t_c)$.  Let $Y_\tau=\ov Y(\tau)$.
By Theorem~\ref{der_jumps} there exists $\dot Y_c\in\R^4$ such that
\begin{equation}
\lim_{t\to t_c^-}\dot{\ov Y}(t) = \dot Y_c\ .
\label{limit}
\end{equation}
In fact relation \eqref{limit} is fulfilled by the
solution of the Cauchy problem\footnote{Here $(\ov{\nabla_Y R})_h^+$ is the
  vector with components $\bigl(\frac{\ov{\partial R}}{\partial
    y_k}\bigr)_h^+$, $k=1\ldots 4$ introduced in Theorem~\ref{der_jumps}.}
\begin{equation}
  \dot{\ov Y}  = -\mathbb{J}_2\,(\ov{\nabla_Y R})_h^+\,,
\hskip 1cm
\ov Y(\tau) = Y_\tau\,,
\label{CPleft}
\end{equation}
which corresponds to the solution of \eqref{cauchy} in the interval $(\tau,t_c)$
and is defined also at the crossing time $t_c$. Let us denote by $Y_c$ its
value for $t=t_c$.
Using again Theorem~\ref{der_jumps} we can extend $\ov Y(t)$ beyond the
crossing singularity by considering the new problem
\begin{equation}
  \dot{\ov Y}  = -\mathbb{J}_2\,(\ov{\nabla_Y R})_h^-\,,
\hskip 1cm
\ov Y(t_c) = Y_c\ .
\label{CPright}
\end{equation}
The solution of \eqref{CPright} fulfils
\begin{equation}
\lim_{t\to t_c^+}\dot{\ov Y}(t) = \dot Y_c + \mathrm{Diff}_h(\ov{\nabla_Y
  R})(\ov\calE(t_c))\ .
\label{deriv_jump}
\end{equation}
The vector field in \eqref{CPright} corresponds to $-\mathbb{J}_2\ov{\nabla_Y
  R }$ on ${\cal W}^-$, thus we can continue the solution outside ${\cal W}$ and
this procedure can be repeated at almost every crossing singularities.
Indeed, the generalized solution is unique provided the evolution
$t\mapsto\ov\calE(t)$ is not tangent to the orbit crossing set $\Sigma$.

\noindent Moreover, if $\det{\cal A}_h = 0$ the extraction of the singularity,
described in Section~\ref{s:extract}, cannot be performed.

In case $\ov\calE(t)\in{\cal W}^-$ for $t\in(\tau,t_c)$ the previous
discussion still holds if we exchange $(\ov{\nabla_Y R})_h^+$ with
$(\ov{\nabla_Y R})_h^-$. In this case \eqref{deriv_jump} becomes
\[
\lim_{t\to t_c^+}\dot{\ov Y}(t) = \dot Y_c - \mathrm{Diff}_h(\ov{\nabla_Y
  R})(\ov\calE(t_c))\ .
\]

\section{Evolution of the orbit distance}
\label{s:dminevol}

We prove that the secular evolution of $\tilde{d}_{min}$ is more
regular than that of the orbital elements in a neighborhood of a
planet crossing time.
We introduce the secular evolution of the distances $\tilde{d}_h$ and of
the orbit distance $\tilde{d}_{min}$:
\begin{equation}
{\ov d}_h(t) = \tilde{d}_h(\barcalE(t))\,,
\hskip 1cm
{\ov d}_{min}(t) = \tilde{d}_{min}(\barcalE(t))\ .
\label{evolution}
\end{equation}
Assume these maps are defined in an open interval containing a crossing time
$t_c$, and suppose $\calE_c=\barcalE(t_c)$ is a non-degenerate crossing
configuration at time $t_c$,
as in Section~\ref{s:extract}.

In the following we shall discuss only the case of $\tilde{d}_h$. The same
result holds for $\tilde{d}_{min}$, taking care of the possible exchange of role
of two local minima $d_h, d_k$ as absolute minimum.

\begin{proposition}
  Let $\ov Y(t)$ be a generalized solution of \eqref{cauchy} and $\ov\calE(t)$
  as in \eqref{pair}, \eqref{astorb}. Assume $t_c\in T(\ov Y)$ is a
  crossing time and $\calE_c = \ov\calE(t_c)$ is a non-degenerate crossing
  configuration with only one crossing point.
Then there exists an interval $(t_a,t_b)$, $t_a<t_c<t_b$
 such that ${\ov d}_h\in C^1((t_a,t_b);\R)$.
\label{regevol}
\end{proposition}
\begin{proof}
  Let the interval $(t_a,t_b)$ be such that $\ov\calE((t_a,t_b))\subset {\cal
    W}$ , where ${\cal W}$ is chosen as in Theorem~\ref{der_jumps}.
We can assume that $\ov\calE(t)\in{\cal W}^+$ for $t\in(t_a,t_c)$,
$\ov\calE(t)\in{\cal W}^-$ for $t\in(t_c,t_b)$ (the proof for the opposite
case is similar).
For $t\in (t_a,t_b)\setminus\{t_c\}$ the time derivative of ${\ov d}_h$ is
\begin{eqnarray*}
\dot{{\ov d}}_h(t) &=&
\nabla_{\calE} \tilde{d}_h (\barcalE(t))
\cdot \dot{\barcalE}(t) =
\nabla_Y\tilde{d}_h(\barcalE(t))\cdot\dot{\ov Y}(t) +
\nabla_{E'}\tilde{d}_h(\barcalE(t))\cdot
\dot{\ov E}'(t)\\
&=&
-\left(\nabla_Y\tilde{d}_h\cdot
\mathbb{J}_2\ov{\nabla_Y R}\right)(\barcalE(t)) +
\nabla_{E'}\tilde{d}_h(\barcalE(t))\cdot
\dot{\ov E}'(t)\ .
%
\end{eqnarray*}
Here $\nabla_{\calE}$, $\nabla_Y$, $\nabla_{E'}$ denote the vectors of partial
derivatives with respect to $\calE, Y, E'$ respectively.
The derivative $\dot{\ov E}'(t)$ exists also for $t=t_c$.  On the
other hand, by Theorem~\ref{der_jumps}, the restrictions of $\ov{\nabla_Y
  R}(\barcalE(t))$ to $t <t_c$ and $t>t_c$ admit two different continuous
extensions to $t_c$.  By \eqref{diff}, since $\tilde{d}_h(\barcalE(t_c)) = 0$, we have
\begin{eqnarray*}
\lim_{t\to t_c^+} \dot{\ov d}_h(t) - \lim_{t\to t_c^-} \dot{\ov d}_h(t) &=&
\left.{\rm Diff}_h\bigl(\ov{\nabla_Y R}\bigr)\cdot \mathbb{J}_2\nabla_Y
\tilde{d}_h\right|_{\calE = \calE_c} \\
&=& \frac{\mu k^2}{\pi\sqrt{\det{\cal A}_h}}
\left.\{\tilde{d}_h, \tilde{d}_h\}_Y\right|_{\calE = \calE_c} = 0\,,
\end{eqnarray*}
where $\{,\}_Y$ is the Poisson bracket with respect to $Y$. Thus the time
derivative of ${\ov d}_h$ exists and is continuous also in $t=t_c$.
\end{proof}


\break

\section{Numerical experiments}
\label{s:numexp}

\subsection{The secular evolution program}

Using a model with 5 planets, from Venus to Saturn, we compute a planetary
ephemerides database for a time span of $50,000$ yrs starting from epoch 0 MJD
(November 17, 1858) with a time step of $20$ yrs.  The computation is
performed using the FORTRAN program {\tt orbit9}, included in the OrbFit
free software\footnote{\url{http://adams.dm.unipi.it/~orbmaint/orbfit/}}.
From this database we can obtain, by linear interpolation, the
evolution of the planetary trajectories at any time in the specified time interval.

We describe the algorithm to compute the solutions of the averaged equations
\eqref{averrhs} beyond the singularity, where $R$ is now
the sum of the perturbing functions $R_i$, $i=1\ldots 5$, each related to a
different planet.  We use a Runge-Kutta-Gauss (RKG)
method to perform the integration: it evaluates the averaged vector field only
at intermediate points of the integration time interval.  When the asteroid
trajectory is close enough to an orbit crossing, then the time step is
decreased to reach the crossing condition exactly.

From Theorem~\ref{der_jumps} we can find two Lipschitz-continuous extensions
of the averaged vector field from both sides of the singular set $\Sigma$.

To compute the solution beyond the singularity we use the explicit formula
\eqref{diff} giving the difference between the two extended vector fields,
either of which corresponds to the averaged vector field on different sides of
$\Sigma$.  We compute the intermediate values of the extended vector field
just after crossing, then we correct these values by \eqref{diff} and use them
as approximations of the averaged vector field in \eqref{averrhs} at the
intermediate points of the solutions, see Figure~\ref{fig:cross_step}.  This
RKG algorithm avoids the computation of the extended vector field at the
singular points, which may be affected by numerical instability.

A difficulty in the application of this scheme is to estimate the size of a
suitable neighborhood ${\cal W}$ of the crossing configuration $\calE_c$
fulfilling the conditions given in Section~\ref{s:extract}.

\begin{figure}[h!]
  \psfragscanon \psfrag{A}{$\overline{Y}_{k-1}$}
  \psfrag{B}{$\overline{Y}_k$} \psfrag{C}{$\overline{Y}_{k+1}$}
  \psfrag{Wp}{${\cal W}^+$}\psfrag{Wm}{${\cal W}^-$}
  \psfrag{sigma}{$\Sigma$} 
  \centerline{\epsfig{figure=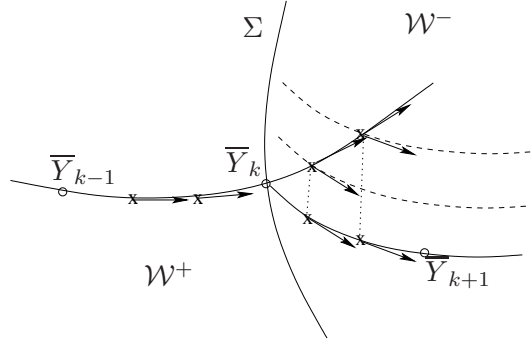,width=7cm}}
  \psfragscanoff 
   \caption{Runge-Kutta-Gauss method and continuation of the solutions
    of equations \eqref{averrhs} beyond the singularity. The crosses
    correspond to the intermediate values.}
\label{fig:cross_step}
\end{figure}

\begin{figure}[h!]
\centerline{\epsfig{figure=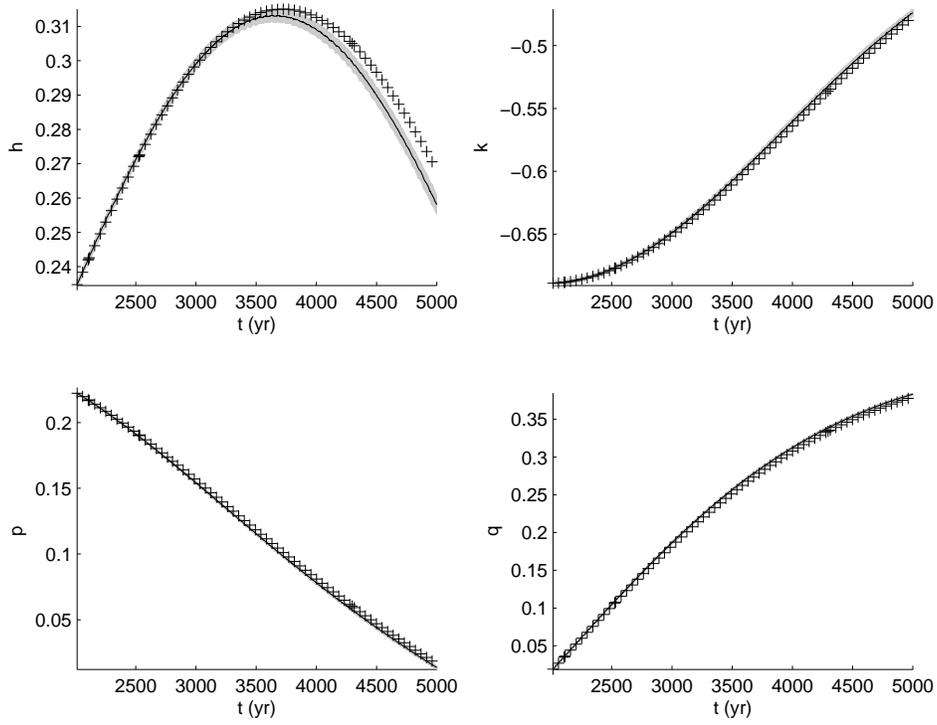,width=12.5cm}}
\caption{Averaged and non-averaged evolutions of asteroid
  $1979$~XB.}
\label{fig:evolcomp_79XB}
\end{figure}

\begin{figure}[h!]
\centerline{\epsfig{figure=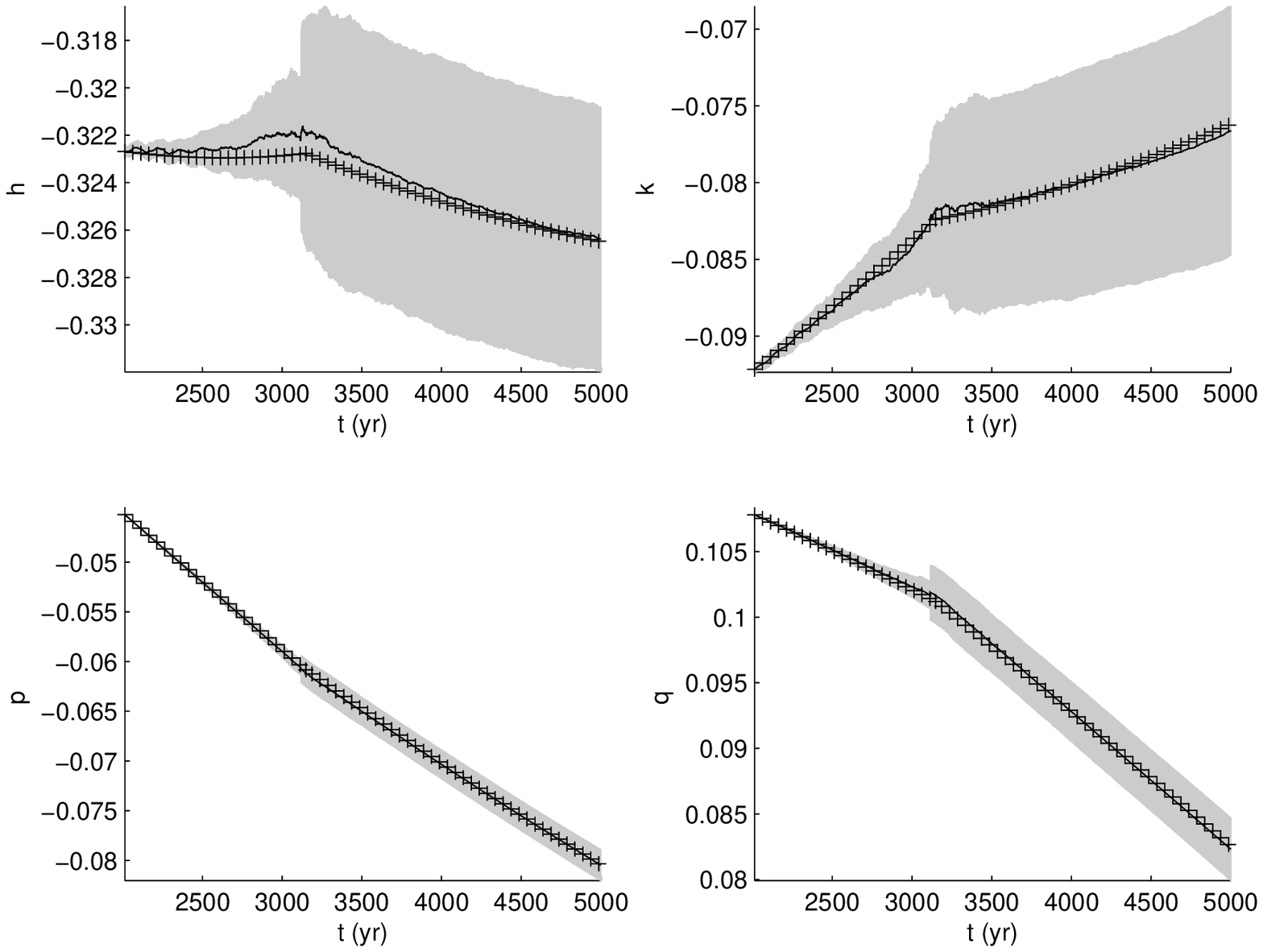,width=12.5cm}}
\caption{Averaged and non-averaged evolutions of asteroid
  1620 (Geographos).}
\label{fig:evolcomp_1620}
\end{figure}

\subsection{Comparison with the solutions of the full equations}
\label{s:compevol}

We performed some tests to compare the solutions of the averaged equations
\eqref{averrhs} with the corresponding components of the solutions of the full
equations \eqref{Heqs}.  Here we show two tests with the asteroids $1979$~XB
and 1620 (Geographos).  We considered the system composed by an asteroid and 5
planets, from Venus to Saturn.  We selected the 8 values $k\pi/4$, with
$k=0\ldots 7$, for the initial mean anomaly of the asteroid and the same for
the planets. Using the program {\tt orbit9}, we performed
the integration of the system with these 64 different initial conditions
(i.e. we selected equal initial phases for all the planets).
Then we considered the arithmetic mean of the four
equinoctial\footnote{We recall that
\[
h=e\sin(\omega+\Omega)\,,\quad k=e\cos(\omega+\Omega)\,,\quad
p=\tan(I/2)\sin(\Omega)\,,\quad q=\tan(I/2)\cos(\Omega)\ .
\]
The equinoctial orbital elements have been introduced in
\cite{broucke_cefola}.} orbital elements $h,k,p,q$ of the asteroid over these evolutions, and compared
them with the results of the secular evolution.  In
Figures~\ref{fig:evolcomp_79XB}, \ref{fig:evolcomp_1620}, we show the results:
the crosses indicate the secular evolution, the continuous curve
is the mean of full numerical one and the gray region represents
the standard deviation from the mean.  The correspondence between the solutions
is good.
During the evolution the distance between the asteroid and the Earth
for some initial conditions attains values of the order of
$10^{-4}$ au for $1620$ (Geographos), and $10^{-3}$ au for $1979$~XB.
In Figure~\ref{fig:evolcomp_1620} the Earth crossing singularity is
particularly evident near the epoch 3000 AD.

Some numerical tests of the theory introduced in \cite{GM98}, with the planets
on circular coplanar orbits, can be found in \cite{comparison}.

\subsection{An estimate of planet crossing times}
\label{s:tcros}

The results of Section~\ref{s:dminevol} can be used to estimate the epoch in
which the orbit of a near-Earth asteroid will cross that of the Earth.
We are interested in particular to study the behavior of those asteroids whose
orbits will cross the Earth in the next few centuries, so that they must have a
small value of $d_{min}$ already at the present epoch.
We can consider, for example, the set of potentially hazardous asteroids
(PHAs), which have $d_{min}\leq 0.05$ au and absolute magnitude $H_{mag} \leq
22$, i.e. they are also large.

In Figure~\ref{fig:dmin} we show 3 different evolutions of the signed orbit
distance $\tilde{d}_{min}$ for the PHA $1979$~XB.  Here we
draw the full numerical (solid line), secular (dashed) and secular linearized
(dotted) evolution of $\dmint$.
\begin{figure}[h]
  \centerline{\epsfig{figure=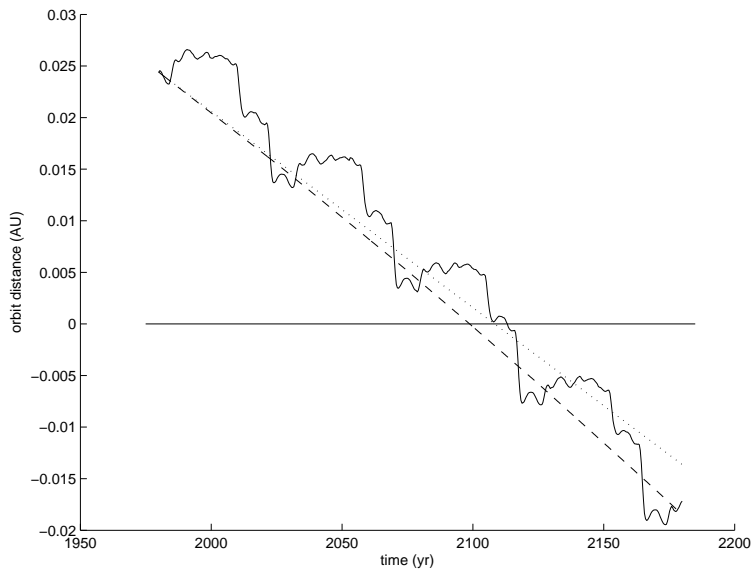,width=10cm}}
  \caption{Different evolutions of $\dmint$ for $1979$~XB: full numerical
    (solid line), secular (dashed) and secular linearized (dotted).}
  \label{fig:dmin}
\end{figure}
By Proposition~\ref{regevol} the linearization of the secular evolution
$\ov{d}_{min}(t)$ can give a good approximation also in a neighborhood of a
crossing time.

We propose a method to compute an interval $J$ of possible crossing times.  We
sample the {\em line of variation} (LOV), introduced in \cite{M&al05}, which
is a sort of `spine' of the confidence region (see also \cite{MG2010}), and
compute the signed orbit distance $\tilde{d}_{min}$ for each virtual asteroid
(VA) of the sample. Then we compute the time derivative of $\bar{d}_{min}$ for
each VA and extrapolate the crossing times by a linear approximation of the
evolution. We set $J = [t_1,t_2]$, with $t_1, t_2$ the minimum and maximum
crossing times obtained (see Figure~\ref{LOV_linear}).  In the computation of
$J$ we take into account a band centered at the Earth crossing line
$d_{min}=0$: in this test the width of the considered band is $2\times
10^{-3}$ au.
\begin{figure}[h]
  \centering
  \centerline{\epsfig{figure=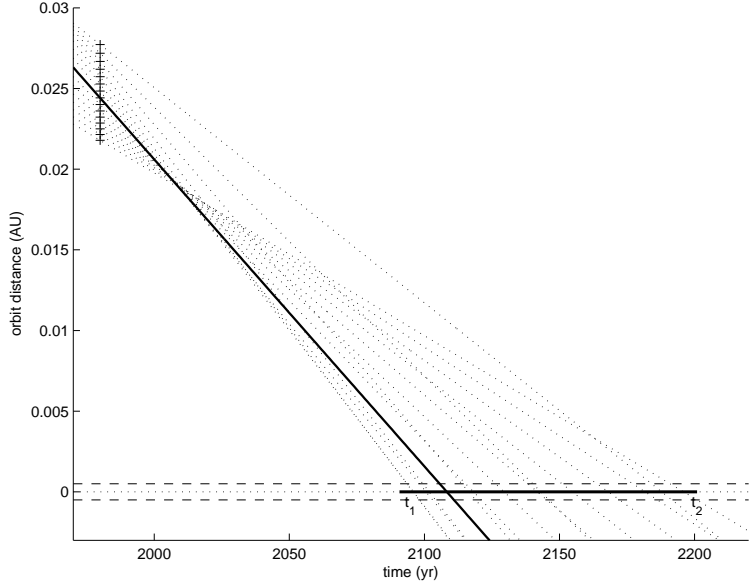,width=10cm}}
  \caption{Computation of the interval $J$ (horizontal solid line) for
    asteroid $1979$~XB. The transversal solid line corresponds to the
    linearized secular evolution of the nominal orbit. The linearized secular
    evolution of the VAs are the dotted lines.}
  \label{LOV_linear}
\end{figure}

We describe a method to assign a probability of occurrence of crossings in a
given time interval, which is related to the algorithm described above.
For each
value of the LOV parameter $s$ we have a VA at a time $t$, so that we can
compute $\bar{d}_{min}(t)$. Thus, using the scheme of Figure~\ref{LOV_linear}
we can define a map $\mathfrak{T}$ from the LOV
parameter line to the time line.  The map $\mathfrak{T}$ gives the crossing
times, using the linearized secular dynamics, for the VAs
on the LOV that correspond to the selected values of the parameter $s$.
Moreover, we have a
probability density function $p(s)$ on the LOV.
Therefore, given an interval $I$ in the time line, we can consider the set
$U_I = \mathfrak{T}^{-1}(I)$ and define the probability of having a crossing
in the time interval $I$ as
\[
P(I) = \int_{U_I}p(s)\;ds\ .
\]

Finally, in Figure~\ref{confronto} we show the corresponding interval $J'$
obtained by computing the secular evolution (without linearization) of the
orbit distance for each VA of $1979$~XB. The sizes of $J$ and $J'$ are almost
equal, but the left extremum of $J'$ is $\sim 10$ years before.
\begin{figure}[h]
  \centering
  \psfragscanon
  \psfrag{t}{$t$}\psfrag{d}{$d$}
  \centerline{\epsfig{figure=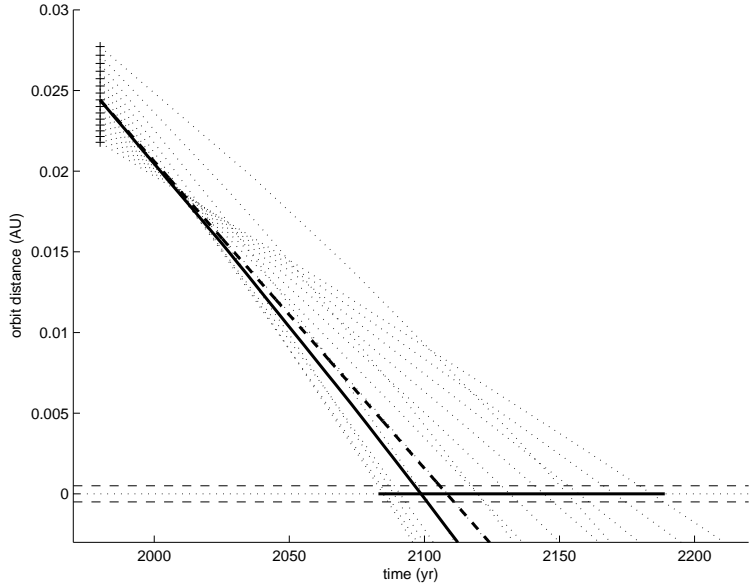,width=10cm}}
  \psfragscanoff
  \caption{Computation of the interval $J'$ (horizontal solid line) for
    asteroid $1979$~XB. The enhanced transversal curves
    refer to the nominal orbit: solid line corresponds to secular evolution,
    linearized is dashed. The dotted curves represent the secular evolution of
    the VAs.}
  \label{confronto}
\end{figure}


\section{Conclusions and future work}

We have studied the double averaged restricted 3-body problem in case
of orbit crossing singularities, improving and completing the results in
\cite{GM98}, \cite{gronchi_celmecIII}.  This
problem is of interest to study the dynamics of near-Earth asteroids
from a statistical point of view, going beyond the Lyapounov times of
their orbits.  We have also proved that generically, in a neighborhood
of a crossing time, the secular evolution of the (signed) orbit distance
is more regular than the averaged evolution of the orbital elements.

The solutions of this averaged problem have been computed by a numerical
method and then compared with the solutions of the full equations in a few
test cases. The results were good enough; however, we expect that the
averaging technique fails in case of mean motion resonances or close
encounters with a planet. We plan to perform numerical experiments with a
large sample of near-Earth asteroids showing different behaviors: this will be
useful to understand the applicability of the averaging technique to the whole
set of NEAs.

\section*{Acknowledgments}
We would like to thank an anonymous referee for his interesting comments, that
allowed us to improve the paper.


\begin{thebibliography}{99}

\bibitem{AKN} (MR1656199)
\newblock V.~I. Arnold, V. V. Kozlov and A. I. Neishtadt,
\newblock ``Mathematical Aspects of Classical and Celestial Mechanics,"
\newblock Springer, 1997.

\bibitem{balukhol} (MR2162789) [10.1007/s10569-004-3207-1]
\newblock R.~V. Baluyev and K.~V. Kholshevnikov,
\newblock \emph{Distance between Two Arbitrary Unperturbed Orbits},
\newblock Cel. Mech. Dyn. Ast., \textbf{91} (2005), 287--300.

\bibitem{bowell}
\newblock E. Bowell and K. Muinonen,
\newblock \emph{Earth-crossing asteroids and comets: Groundbased search strategies},
\newblock in ``Hazards due to Comets \& Asteroids" (Ed. T. Gehrels), Tucson: The University of Arizona Press, (1994), 149--197.

\bibitem{broucke_cefola}
\newblock R.~A. Broucke and P.~J. Cefola,
\newblock \emph{On the equinoctial orbit elements},
\newblock Cel. Mech. Dyn. Ast., \textbf{5} (1972), 303--310.

\bibitem{gronchi_celmecIII} (MR1956518) [10.1023/A:1020178613365]
\newblock G.~F. Gronchi,
\newblock \emph{Generalized averaging principle and the secular evolution of planet crossing orbits},
\newblock Cel. Mech. Dyn. Ast., \textbf{83} (2002), 97--120.

\bibitem{gronchi02} (MR1924415) [10.1137/S1064827500374170]
\newblock G.~F. Gronchi,
\newblock \emph{On the stationary points of the squared distance between two ellipses with a common focus},
\newblock SIAM Journ. Sci. Comp., \textbf{24} (2002), 61--80.

\bibitem{gronchi05} (MR2186811) [10.1007/s10569-005-1623-5]
\newblock G.~F. Gronchi,
\newblock \emph{An algebraic method to compute the critical points of the distance function between two Keplerian orbits},
\newblock Cel. Mech. Dyn. Ast., \textbf{93} (2005), 297--332.

\bibitem{GM98} (MR1675692) [10.1023/A:1008315321603]
\newblock G.~F. Gronchi and A. Milani,
\newblock \emph{Averaging on Earth-crossing orbits},
\newblock Cel. Mech. Dyn. Ast., \textbf{71} (1998), 109--136.

\bibitem{GM01}
\newblock G.~F. Gronchi and A. Milani,
\newblock \emph{Proper elements for Earth crossing asteroids},
\newblock Icarus, \textbf{152} (2001), 58--69.

\bibitem{comparison}
\newblock G.~F. Gronchi and P. Michel,
\newblock \emph{Secular orbital evolution, proper elements and proper frequencies for near-earth asteroids: A comparison between semianalytic theory and numerical integrations},
\newblock Icarus, \textbf{152} (2001), 48--57.

\bibitem{gronchi_tommei} (MR2291871) [10.3934/dcdsb.2007.7.755]
\newblock G.~F. Gronchi and G. Tommei,
\newblock \emph{On the uncertainty of the minimal distance between two confocal Keplerian orbits},
\newblock Discrete Contin. Dyn. Syst. Ser. B, \textbf{7} (2007), 755--778.

\bibitem{kholvass} (MR1750216) [10.1023/A:1008312521428]
\newblock K.~V. Kholshevnikov and N. Vassiliev,
\newblock \emph{On the distance function between two keplerian elliptic orbits},
\newblock Cel. Mech. Dyn. Ast., \textbf{75} (1999), 75--83.

\bibitem{kinonakai07} (MR2317260) [10.1007/s10569-007-9069-6]
\newblock H. Kinoshita and H. Nakai,
\newblock \emph{General solution of the Kozai mechanism},
\newblock Cel. Mech. Dyn. Ast., \textbf{98} (2007), 67--74.

\bibitem{kozai62}
\newblock Y. Kozai,
\newblock \emph{Secular perturbation of asteroids with high inclination and eccentricity},
\newblock Astron. Journ., \textbf{67} (1962), 591--598.

\bibitem{lidov}
\newblock M.~L. Lidov,
\newblock \emph{The evolution of orbits of artificial satellites of planets under the action of gravitational perturbations of external bodies},
\newblock Plan. Spa. Sci., \textbf{9} (1962), 719--759.

\bibitem{M&al05}
\newblock A. Milani, S.~R. Chesley, M.~E. Sansaturio, G. Tommei and G. Valsecchi,
\newblock \emph{Nonlinear impact monitoring: Line of variation searches for impactors},
\newblock Icarus, \textbf{173} (2005), 362--384.

\bibitem{MG2010} (MR2778686)
\newblock A. Milani and G. F. Gronchi,
\newblock ``Theory of Orbit Determination,"
\newblock Cambridge Univ. Press, 2010.

\bibitem{resret}
\newblock G. B. Valsecchi, A. Milani, G. F. Gronchi and S. R. Chesley,
\newblock \emph{Resonant returns to close approaches: Analytical theory},
\newblock Astron. Astrophys., \textbf{408} (2003), 1179--1196.

\bibitem{whipple95}
\newblock A. Whipple,
\newblock \emph{Lyapunov times of the inner asteroids},
\newblock Icarus, \textbf{115} (1995), 347--353.
\end{thebibliography}
\end{document}